\documentclass{amsart}
\usepackage[utf8]{inputenc}
\usepackage{amsmath,amssymb,amsthm}
\newtheorem{theorem}[]{Theorem}
\newtheorem{lemma}[]{Lemma}
\newtheorem{corollary}[]{Corollary}

\theoremstyle{definition}
\newtheorem{definition}[]{Definition}

\theoremstyle{remark}
\newtheorem{question}[]{Question}

\usepackage{enumitem}    
\setlist[enumerate]{
    label=\textnormal{(\arabic*)}, 
}

\newlist{enumloc}{enumerate}{1}
\newlist{enumglob}{enumerate}{1}
\setlist[enumloc]{
    label=\textnormal{(\textit{L}\arabic*)}, 
  }
\setlist[enumglob]{
  label=\textnormal{(\textit{G}\arabic*)}, 
}
\newcommand{\defiff}{\stackrel{\mbox{\footnotesize \textrm{def}}}{\iff}}

\title{How to escape Tennenbaum's theorem}
\author{Fedor Pakhomov}\thanks{Research of Fedor Pakhomov was supported by FWO grant G0F8421N}
 \address{Vakgroep Wiskunde: Analysis, Logic and Discrete Mathematics, 
Ghent University,
Krijgslaan 281,
B9000~~Ghent,
Belgium\newline
and Steklov Mathematical Institute of Russian Academy of Sciences, 
Gubkina 8,
119991 Moscow, Russia}
\email{fedor.pakhomov@ugent.be}

\date{September 2022}

\begin{document}

\maketitle
\begin{abstract}
    We construct a theory definitionally equivalent to first-order Peano arithmetic $\mathsf{PA}$ and a non-standard computable model of this theory. The same technique allows us to construct a theory definitionally equivalent to Zermelo-Fraenkel set theory $\mathsf{ZF}$ that has a computable model.
\end{abstract}
\section{Introduction}
A celebrated result of Stanley Tennenbaum \cite{tennenbaum1959non} states there are no computable non-standard models of first-order Peano arithmetic $\mathsf{PA}$. There have been significant amount of research about generalizations of the theorem  (see in particular \cite{shepherdson1964non,wilmers1985bounded,berarducci1996recursive,godziszewski2017computable}) and about its philosophical implications (see in particular \cite{halbach2005computational,button2012philosophical,quinon122007intended}). However, to the best of the author's knowledge most of the research in this direction was focused on the arithmetical signature. 

The theory $\mathsf{PA}$ is widely used in mathematical logic as a theory allowing to formalize reasoning by mathematical induction about finite objects of various kinds. But since $\mathsf{PA}$ directly speaks only about natural numbers, the finite objects that aren't natural numbers are coded there by natural numbers. However, the choice of natural numbers as the base kind of finite objects that is used to code other kinds appears to be just a consequence of the historical development of the field. There are some works investigating formal theories using some other base kinds of finite objects, in particular theories that use hereditary finite sets \cite{sazonov1997bounded,kaye2007interpretations} and theories using binary strings \cite{ferreira1988polynomial}.

In general, the property of two theories to be the same up to the choice of base notions is captured by the notion of \emph{bi-interpretability} of theories that was introduced by Ahlbrandt and Ziegler \cite{ahlbrandt1986quasi}. In this paper we use an even stricter notion of sameness of theories called \emph{definitional equivalence} (also known as \emph{synonymy}) that was introduced by de Bouvère \cite{de1965logical}. We note that it was proved by Friedman and Visser \cite{friedman2014bi} that for a wide range of theories (including $\mathsf{PA}$ and $\mathsf{ZF}$) bi-interpretations could be transformed to definitional equivalences.

One particular theory that is definitionally equivalent to $\mathsf{PA}$ is an appropriate theory of hereditary finite sets that we denote $\mathsf{ZF}_{\mathsf{fin}}^+$ (see a paper by Kaye and Wong \cite{kaye2007interpretations}). It was proved by Enayat, Schmerl, and Visser that there are no non-standard computable models of $\mathsf{ZF}_{\mathsf{fin}}^+$. There is an even stronger result by Godziszewski and Hamkins \cite{godziszewski2017computable} that there are no computable quotient presentations of non-standard models of $\mathsf{ZF}_{\mathsf{fin}}^+$.

In this paper we show that it is possible to construct a theory definitionally equivalent to $\mathsf{PA}$ with a non-standard computable model (Corollary \ref{computable_model_of_PA}). We obtain this as a corollary of a more general result. We construct certain theory $T_0$ that is definitionally equivalent to the theory $\mathsf{ZF}_{-\mathsf{inf}}^+$ (that is $\mathsf{ZF}$ with removed axiom of infinity and added axiom of transitive closure). We show that any consistent c.e. extension of $T_0$ has a computable model (Theorem \ref{main_theorem}). To get Corollary \ref{computable_model_of_PA} from this we use the fact that  $\mathsf{ZF}_{-\mathsf{inf}}^+\subseteq \mathsf{ZF}_{\mathsf{fin}}^+$. And since $\mathsf{ZF}_{-\mathsf{inf}}^+\subseteq \mathsf{ZF}$, Theorem \ref{main_theorem} also allows to construct a theory with computable models that is definitionally equivalent to $\mathsf{ZF}$ (Corollary \ref{computable_model_of_ZF}). Finally we establish a limit of the phenomenon of computable non-standard models in alternative signatures and show that no theory definitionally equivalent to the true arithmetic has a computable non-standard model (Theorem \ref{true_arithmetic}).

\section{Preliminaries}\label{preliminaries}
\begin{definition}
A computable model of a finite signature is a model $M$, whose domain is a computable set of naturals and the interpretations in $M$ of all functions and predicates are computable functions and sets of tuples, respectively. 
\end{definition}

The theories that we consider in this paper are first-order theories with equality.

\begin{definition}
A theory $T$ is a definitional extension of a theory $U$ if
\begin{enumerate}
    \item the signature of $T$ extends the signature of $U$;
    \item $T\vdash \varphi\iff U\vdash\varphi$, for all $U$-sentences $\varphi$;
    \item for any $T$-predicate $P(\vec{x})$ there is a $U$-formula $\mathsf{D}_P(\vec{x})$ such that $$T\vdash \forall\vec{x}\;( P(\vec{x})\mathrel{\leftrightarrow} \mathsf{D}_P(\vec{x}));$$
    \item for any $T$-function $f(\vec{x})$ there is a $U$-formula $\mathsf{D}_f(\vec{x},y)$ such that $$T\vdash \forall\vec{x}\;( f(\vec{x})=y\mathrel{\leftrightarrow} \mathsf{D}_f(\vec{x},y)).$$
\end{enumerate}
Naturally for a given theory $U$ we obtain a definitional extension $T$ by giving first-order definitions $\mathsf{D}_P$,$\mathsf{D}_f$ of all additional predicate and function symbols of the designated signature of $T$.
\end{definition}

\begin{definition}
Theories $T$ and $U$ with disjoint signatures are definitionally equivalent if there is a theory $V$ that is both a definitional extension of $T$ and a definitional extension of $U$. In general theories $T$ and $U$ are definitionally equivalent if so are theories $T'$ and $U'$ obtained from $T$ and $U$ by renaming of the symbols of their signatures to make them disjoint.
\end{definition}

\begin{definition}Let $\mathsf{ZF}_{-\mathsf{inf}}$ be the theory axiomatized by axioms of extensionality, pair, union, power set, and regularity as well as the schemes of separation and replacement. Theory $\mathsf{ZF}_{\mathsf{fin}}$ extends $\mathsf{ZF}_{-\mathsf{inf}}$ by the negation of the axiom of infinity. The axiom of transitive closure $\mathsf{TC}$ asserts that any set is contained in a transitive set. We denote as $\mathsf{ZF}_{-\mathsf{inf}}^+$ and $\mathsf{ZF}_{\mathsf{fin}}^+$ the extensions  by $\mathsf{TC}$ of $\mathsf{ZF}_{-\mathsf{inf}}$ and $\mathsf{ZF}_{\mathsf{fin}}$, respectively.\end{definition}

The key feature of  $\mathsf{ZF}_{-\mathsf{inf}}^+$ is that it proves that any set lies at some level of von Neumann hierarchy $V_\alpha$.

By the result of Kaye and Wong \cite{kaye2007interpretations} first-order Peano arithmetic $\mathsf{PA}$ is definitionally equivalent to the theory $\mathsf{ZF}_{\mathsf{fin}}^+$. For the reader convenience we sketch the proof of definitional equivalence. In $\mathsf{PA}$ Ackerman's membership predicate $n\in_{\mathsf{ack}} m$ expresses that $s_n=0$,  for any presentation of $m$ in the form $m=\sum\limits_{i\le k}s_i2^k$, where  $k\ge n$ and $s_0,\ldots,s_k\in\{0,1\}$. We obtain a definitional extension $T$ of $\mathsf{PA}$ by adding to $\mathsf{PA}$ the predicate symbol $x\in y$ and the axiom $\forall x,y\;(x\in_{\mathsf{ack}} y\mathrel{\leftrightarrow} x\in y)$. It is easy to see that extension contains $\mathsf{ZF}_{\mathsf{fin}}^+$. In $\mathsf{ZF}_{\mathsf{fin}}^+$ we define by recursion over $\in$ a bijection $\mathsf{iack}\colon V\to \mathsf{Ord}$  (here $V$ is the same as the class of all hereditary finite sets and $\mathsf{Ord}$ is the class of all finite ordinals)
$$\mathsf{iack}(x)=\sum_{y\in x} 2^{\mathsf{iack}(y)}.$$
This allows us to produce set-theoretic definitions for the arithmetical signature $0_{\mathsf{iack}}=\emptyset$, $1_{\mathsf{iack}}=\{\emptyset\}$, $x+_{\mathsf{iack}}y=\mathsf{iack}^{-1}(\mathsf{iack}(x)+\mathsf{iack}(y))$, and $x\times_{\mathsf{iack}}y=\mathsf{iack}^{-1}(\mathsf{iack}(x)\mathsf{iack}(y))$. This definitions produce a definitional extension $T'$ of $\mathsf{ZF}_{\mathsf{fin}}^+$ to the language of $\mathsf{PA}$. Clearly, $T'$ proves all the axioms of $\mathsf{PA}$ and the fact that $\in_{\mathsf{ack}}$ coincides with $\in$. Thus $T'$ proves all axioms of $T$. On the other hand clearly $T$ proves that $0_{\mathsf{iack}},1_{\mathsf{iack}},+_{\mathsf{iack}},$ and $\times_{\mathsf{iack}}$ coincide with $0,1,+,$ and $\times$, respectively. Hence $T$ proves all axioms of $T'$. Therefore $T=T'$ is a joint definitional extension of $\mathsf{PA}$ and $\mathsf{ZF}_{\mathsf{fin}}^+$.

\section{A ternary alternative for the membership predicate}
The goal of this section is to define a theory $T_0$ definitionally equivalent to $\mathsf{ZF}_{-\mathsf{inf}}^+$ that we latter prove to have computable models for any of its consistent c.e. extensions.

\begin{definition}
An \emph{$S$-structure} is a finite structure in the language with one ternary predicate $S$. An \emph{$(S,\in)$-structure} is a finite structure $A$ in the language with a ternary predicate $S$ and a binary predicate $\in$ such that for any $a,b,c\in A$, if $A\models a\in b$, then $A\models S(a,b,c)$.
\end{definition}
\begin{definition}For structures $A,B$ having $S$ in their signature an $S$-\emph{embedding} $f\colon A\to B$ is a map from the domain of $A$ to the domain of $B$ such that $$A\models S(a,b,c)\iff B\models S(f(a),f(b),f(c))\text{, for any }a,b,c\in A.$$ For structures $A,B$ having $S$ and $\in$ in their signatures, an $(S,\in)$-\emph{embedding} $f\colon A\to B$ is a map from the domain of $A$ to the domain of $B$ that is an $S$-embedding such that $$A\models a\in b\iff B\models f(a)\in f(b)\text{, for any }a,b\in A.$$ For $S$-structures ($(S,\in)$-structures) $A,B$ we say that $B$ \emph{extends} $A$ and write $A\subseteq B$ if the domain of $A$ is a subset of the domain of $B$ and the identity function from the domain of $A$ to the domain of $B$ is an $S$-embedding ($(S,\in)$-embedding) of $A$ into $B$.

\end{definition}
\begin{definition}Working in $\mathsf{ZF}_{-\mathsf{inf}}^+$ we call an $(S,\in)$-structure $A$ an $\in$-\emph{absolute} structure, if $$a\in b\iff A\models a\in b\text{, for all }a,b\in A.$$ 
\end{definition}

\begin{definition}
We say that an $(S,\in)$-structure $B$ \emph{neutrally extends} $A$ by an element $v$ and write $A\lessdot_v B$ if the domain of $B$ contains exactly one additional element $v$ and $B\models \forall x(x\not\in v\land v\not\in x)$. 
\end{definition}

We are working in $\mathsf{ZF}_{-\mathsf{inf}}^+$ to define a ternary relation $\mathsf{S}$. For this we will define by transfinite recursion on $\alpha$ ternary relations $\mathsf{S}_\alpha$ on $V_{6\alpha}$. Then will put $\mathsf{S}$ to be the union of $\bigcup\limits_{\alpha\in\mathsf{Ord}}\mathsf{S}_\alpha$.

The relation $\mathsf{S}_0$ is empty and for limit $\lambda$ we put $\mathsf{S}_\lambda=\bigcup\limits_{\alpha<\lambda} \mathsf{S}_\alpha$.

Now let us define $\mathsf{S}_{\alpha+1}$ from $\mathsf{S}_\alpha$. We put  $$(a,b,c)\in S_{\alpha+1}\defiff (a,b,c)\in S_\alpha\text{, for }a,b,c\in V_{6\alpha}.$$ We consider all the pairs of $(S,\in)$-structures $A,B$ such that $A\lessdot_v B$, for some $v$, the domains of $A$ and $B$ are included in $V_{6\alpha}$, $A$ is $\in$-absolute, and $S^A$ coincides with the restriction of $\mathsf{S}_\alpha$ to the domain of $A$. We consider $f_{\alpha,A,B}\colon B\to V_{6\alpha+6}$ that keeps all $a\in A$ in place and maps the unique $v\in B\setminus A$ to $(6\alpha+3, (A,B))$, where we denote by $(x,y)$ the Kuratowski pair $\{\{x\},\{x,y\}\}$. We put $$(f_{\alpha,A,B}(a),f_{\alpha,A,B}(b),f_{\alpha,A,B}(c))\in S_\alpha\defiff (a,b,c)\in B\text{, for all }a,b,c\in B.$$ For the rest of the triples $a,b,c\in V_{6\alpha+6}$ we always put $(a,b,c)\in S_{\alpha+1}$.

We verify the correctness of the definition of $S_{\alpha+1}$ by induction on $\alpha$. We need to check that when defining $S_{\alpha+1}$ we haven't assigned some triple $(a,b,c)$ simultaneously to be an element of $S_{\alpha+1}$ and not to be an element of $S_{\alpha+1}$. The definition above immediately implies that for all the considered triples $A\lessdot_v B$ the sets $f_{A,B,\alpha}(v)$ are pairwise distinct. Also, it is easy to see that a set $f_{A,B,\alpha}(v)$ always has the rank $6\alpha+5$ (i.e. $f_{A,B,\alpha}(v)\in V_{6\alpha+6}$, but $f_{A,B,\alpha}(v)\not\in V_{6\alpha+5}$). Thus the definition of $S_{\alpha+1}$ doesn't have collisions. Also note that sets $f_{A,B,\alpha}(v)$ clearly don't have sets from $V_{6\alpha}$ as their elements. Hence the maps $f_{\alpha,A,B}$ are $(S,\in)$-embeddings of $B$ into the class-size structure $(V;\mathsf{S},\in)$.

\begin{definition}
Let $T_1$ be the definitional extension of $\mathsf{ZF}_{-\mathsf{inf}}^+$, where we add a ternary predicate $S$ together with the defining axioms $S(x,y,z)\mathrel{\leftrightarrow}\mathsf{S}(x,y,z)$. Let $T_0$ be the restriction of $T_1$ to the language with just $S$.
\end{definition}

\begin{definition} We will be interested in $(S,\in)$-structures $A$ whose domain consists only of constant symbols. For structures of this form we denote as $\mathsf{Diag}(A)$ the atomic diagram of $A$, i.e. the set of all true in $A$ atomic sentences and negations of atomic sentences that additionally to predicates $=,S,\in$ could use constants from $A$.\end{definition}

Using the fact that verifiable in  $\mathsf{ZF}_{-\mathsf{inf}}^+$ the maps  $f_{\alpha,A,B}$ are $(S,\in)$-embeddings of $B$ into the class-size structure $(V,\mathsf{S},\in)$ we get the following lemma.
\begin{lemma}\label{univ} For $(S,\in)$-structures $A\lessdot_v B$ whose domains consist of constants  $$T_0\vdash \exists x\Big(\bigwedge \mathsf{Diag}(A)\to\big(\bigwedge \mathsf{Diag}(B)\big)[v\mapsto x]\Big).$$\end{lemma}
From the definition of $\mathsf{S}$ we easily get the following lemma.
\begin{lemma}\label{only_good} Suppose $C$ is a finite set of constants. Let $\mathbf{S}(C)$ be the set of all $(S,\in)$-structures with the domain $C$. Then
  $$T_1\vdash \bigwedge\limits_{c_1,c_2\in C,c_1\ne c_2}c_1\ne c_2\to \bigvee\limits_{A\in \mathsf{S}(C)}\bigwedge\mathsf{Diag}(A).$$
\end{lemma}
\begin{lemma}\label{in_def}
$T_1\vdash x\in y\mathrel{\leftrightarrow} \forall z\; S(x,y,z)$. 
\end{lemma}
\begin{proof}
We reason in $T_1$. The implication $x\in y\to \forall z\; \mathsf{S}(x,y,z)$ is immediate from the definition of $\mathsf{S}$. To prove the inverse implication we assume that $a\not\in b$ and find $c$ such that $\lnot S(a,b,c)$. Consider the two element $(S,\in)$ structure $A$ with the domain $\{a,b\}$ that is absolute for both $S$ and $\in$. Clearly there is some $B\gtrdot_v A$ with $B\not \models S(a,b,v)$. We find large enough $\alpha$ so that the domains of both $A$ and $B$ are contained in $V_{6\alpha}$. Since $f_{\alpha,A,B}$ is an $(S,\in)$-embedding of $B$ into $(V,S,\in)$, we see that by letting $c=f_{\alpha,A,B}(v)$ we get $\lnot S(a,b,c)$.\end{proof}

Immediately from Lemma \ref{in_def} we get the following.
\begin{corollary}Theories $T_0$ and  $\mathsf{ZF}_{-\mathsf{inf}}^+$ are definitionally equivalent.
\end{corollary}

\section{Computable models for extensions of $T_0$}
In this section we prove Theorem \ref{main_theorem} that is our main technical result.

We note that our proof of Theorem \ref{main_theorem} doesn't rely on the details of the definition of $T_0$. The only assumptions that we need are
\begin{enumerate}
\item $T_0$ is a c.e. theory in the signature with one ternary predicate $S$;
\item $T_1$ is a definitional extension of $T_0$ by an additional membership predicate $\in$;
\item $T_0$ satisfies Lemma \ref{univ};
\item $T_1$ satisfies Lemma \ref{only_good}.
\end{enumerate}

\begin{theorem}\label{main_theorem}Any consistent c.e. extension $T$ of $T_0$ has a computable model.\end{theorem}
\begin{proof}
We define theory $T'$. The signature of $T'$ is the least signature extending the signature of $T_1$ by Henkin constants $c_\varphi$, for all $T'$-formulas $\varphi(x)$ without other free variables, where $x$ is a fixed variable. The axioms of $T'$ are:
\begin{enumerate}
    \item all the theorems of $T$;
    \item all the theorems of $T_1$;
    \item Henkin axioms $\exists x\;\varphi(x)\to \varphi(c_\varphi)$.
\end{enumerate}
Clearly, $T'$ is a consistent c.e. extension of $T$.

Now we fix a $\Delta_2$ completion $U$ of $T'$. For the reader convenience we note that $U$ with this properties could be defined in the following standard manner. We fix an enumeration of all $T'$-sentences $\varphi_0,\varphi_1,\ldots$. We define finite sets $K_i$ of $T'$-sentences by recursion on $i$. We put  $K_0=\emptyset$. We put $K_{i+1}=K_i\cup \{\varphi_i\}$, if $T'+K_i\cup \{\varphi_i\}$ is consistent. And we put $K_{i+1}=K_i$, otherwise. We put $U=\bigcup\limits_{i<\omega} K_i$. Clearly, $U$ is a completion of $T'$. Since we could check consistencies of a c.e. theory recursively in $0'$, the sequence $K_i$ is computable relative to $0'$. To check if $\varphi_i\in U$ it is sufficient to check if $\varphi_i\in K_{i+1}$, hence $U$ is computable relative to $0'$, i.e. it is $\Delta_2$.

As usual we have the Henkin model $M'$ of $T'$, i.e. the model whose domain consists of equivalence classes of Henkin constants under the $U$-provable equality, where for atomic formulas $\varphi(x_1,\ldots,x_n)$ we put $$M'\models \varphi([c_1],\ldots,[c_n])\defiff U\vdash \varphi(c_1,\ldots,c_n).$$ Our goal for the rest of the proof will be to find a computable isomorphic copy $M$ of the $S$-part of $M'$.

For this we will construct a computable sequence of $S$-structures $M_0,M_1,\ldots$, a computable sequence of  $(S,\in)$-structures $D_0,D_1,\ldots$, and a computable sequence of maps $(g_i\colon D_i\to M_i)_{i<\omega}$. They will satisfy the following local properties:
\begin{enumloc}
\item \label{l1}$(M_i)_{i<\omega}$ is a sequence of expanding $S$-structures, whose domains are initial segments of naturals;
\item \label{l2}$D_i$ are  $(S,\in)$-structures consisting of Henkin constants;
\item \label{l3}$g_i\colon D_i\to M_i$ are $S$-embeddings.
\end{enumloc}
Additionally our construction will guarantee the following global properties:
\begin{enumglob}
    \item  \label{g1}For any natural $k$ there is $n$ such that $g_i^{-1}(k)$ is defined and is the same, for all $i\ge n$; we denote this limit value of $g_i^{-1}(k)$ as $t(k)$.
    \item  \label{g2}For any naturals $k_1,k_2,k_3\in M$ we have $$U\vdash S(t(k_1),t(k_2),t(k_3))\iff M\models S(k_1,k_2,k_3).$$
    \item  \label{g3}For any Henkin constant $c$ there is $k$ such that $U\vdash c=t(k)$.
\end{enumglob}
We then put $M$ to be the union of all $M_i$. From the conditions above it is easy to see that $M$ is a computable isomorphic copy of the $S$-part of $M'$, whose domain is the whole set of naturals. For the rest of the proof we are constructing sequences $(M_i)_{i<\omega}$, $(D_i)_{i<\omega}$, and $(g_i)_{i<\omega}$ with the desired properties.

Since $U$ is $\Delta_2$ there is a computable sequence of finite sets of $U$-sentences $u_0,u_1,\ldots$ such that 
$$U\vdash \varphi\iff \exists i\forall j>i(\varphi\in u_j)\text{, for any $U$-sentence }\varphi.$$

Using the explicit definition of $U$ that we mentioned above we could give a simple explicit construction of $(u_i)_{i<\omega}$. Namely we define computable family $(K)_{i,j}$. We put $K_{0,j}=\emptyset$. We put $K_{i+1,j}=K_{i,j}\cup\{\varphi_i\}$ if the Turing machine searching for a contradiction in $T'+K_{i,j}\cup \{\varphi_i\}$ doesn't find a contradiction after $j$ steps and we put $K_{i+1,j}=K_{i,j}$, otherwise. We define $u_i$ to be equal to $K_{i,i}$. 

We construct a uniformly computable sequence of computable sets of $U$-sentences $U_0,U_1,\ldots$ such that
\begin{enumerate}
\item $U\vdash \varphi\iff \exists i\forall j>i(\varphi\in U_j)$.
\item \label{U_i_closure} Suppose $A\lessdot_v B$ and $A$ consists only of Henkin constants. Let $e(A,B)$ be the Henkin constant $c_{\varphi}$, where $\varphi(x)$ is $\mathsf{Diag}(A)\to\mathsf{Diag}(B)[v\mapsto x]$. Then $$U_i\supseteq \mathsf{Diag}(A)\Rightarrow U_i\supseteq \mathsf{Diag}(B)[v\mapsto e(A,B)].$$
\end{enumerate}

We produce $U_i$ by starting from $u_i$ and enlarging it by closing it under the rules prescribed by \ref{U_i_closure}. Using Lemma \ref{univ} it is trivial to check that $U_i$ produced like this are uniformly computable and that their limit is indeed equal to $U$.


We fix an ordering $<$ of Henkin constants with the order type $\omega$. 

Together with the structure $D_i$ and injection $g_i$ we will construct a linear ordering $\prec_i$ on the structure $D_i$, and a sequence $p_{i,0}\prec_i\ldots\prec_i p_{i,n_i-1}$ of elements of $D_i$. We also formally add to the sequence element $p_{i,n_i}=\infty$. For $a\in D_i\cup\{\infty\}$ we denote as $D_i{\upharpoonright\prec_i} a$ the substructure of $D_i$ consisting just of $b\prec_i a$ and we denote as $g_i{\upharpoonright\prec_i} a$ the restriction of $g_i$ to $D_i{\upharpoonright\prec_i} a$ (we also use the analogous notation for $a\in D_i$ and $\preceq$). The order $\prec_i$ and the sequence $(p_{i,j})_{j<n_i}$ will satisfy additional local properties:
\begin{enumloc}[resume*]
\item \label{l4}$p_{i,0}<p_{i,1}<\ldots<p_{i,n_{i}-1}$;
\item \label{l5}$U_i\supseteq \mathsf{Diag}(D_i)$;
\item \label{l6}all $a\in D_i$ that are not in the sequence $(p_{i,j})_{j<n_i}$ are of the form $a=e(D_i{\upharpoonright\prec_i} a,D_i{\upharpoonright\preceq_i} a)$. 
\end{enumloc}

Now we are ready to recursively define computable sequences $$(M_i)_{i<\omega},(D_i)_{i<\omega},(g_i)_{i<\omega}.$$ We put $M_0$ to be empty, this leaves us no choice, but to put $D_0, g_0$ and $\prec_0$ to be empty, and to put $n_0=0$. Now we define $M_{i+1},D_{i+1},g_{i+1},\prec_{i+1},$ and the sequence $(p_{i+1,j})_{j<n_{i+1}}$ assuming that we already defined all this objects for the step $i$. We consider three case: 
\begin{enumerate}
    \item \label{step_o1} for some $j<n_i$ either $\mathsf{Diag}(D_i{\upharpoonright\prec_i} p_{i,j+1})\not\subseteq U_{i+1}$ or there is an $(S,\in)$-structure $\dot D_i$ extending $D_{i}{\upharpoonright\prec_i} p_{i,j}$ by exactly one Henkin constant $\dot p_{i,j}$ such that  $\mathsf{Diag}(\dot D_i)\subseteq U_{i+1}$ and $\dot p_{i,j}<p_{i,j}$;
    \item \label{step_o2}neither \ref{step_o1} nor \ref{step_o2} and there is an  $(S,\in)$-structure $\dot D_i$ extending $D_{i}$ by exactly one Henkin constant $c$ such that $\mathsf{Diag}(\dot D_i)\subseteq U_{i+1}$ and there are at most $i$ Henkin constants $<$-below $c$;
    \item \label{step_o3} neither \ref{step_o1} nor \ref{step_o2} hold.
\end{enumerate}   
In the case \ref{step_o3} we simply put $M_{i+1}=M_i$, $D_{i+1}=D_i$, $g_{i+1}=g_i$, $\prec_{i+1}=\prec_i$, $n_{i+1}=n_i$, and $p_{i+1,j}=p_{i,j}$, for $j<n_i$. In the case \ref{step_o1} we choose least $j$ witnessing \ref{step_o2} and then we put $M_{i+1}=M_i$, $n_{i+1}=j$, $D_{i+1}=D_i{\upharpoonright\prec_i} p_{i,j}$, $g_{i+1}=g_i{\upharpoonright}  D_{i+1}$, and $p_{i+1,l}=p_{i,l}$, for $l<n_{i+1}$. 

Now let us describe what we do in the case \ref{step_o2}. We choose $<$-least $c$ and some $\dot D_i$ witnessing \ref{step_o2} (we make the choice of $\dot D_i$ in a computable way). 

The domain of $M_{i+1}$ expands the domain of $M_i$ by one additional natural. Let $\dot g_i\colon \dot D_{i}\to M_{i+1}$ extend $g_i$ by putting $\dot g_i(c)$ to be equal to the newly added natural in the domain of $M_{i+1}$. For triples $a,b,c\in M_i$ we put $$M_{i+1}\models S(a,b,c)\defiff M_i\models S(a,b,c)$$ and for triples $a,b,c\in \dot D_i$ we put $$M_{i+1}\models S(\dot g_i(a),\dot g_i(b),\dot g_i(c))\defiff \dot D_i\models S(a,b,c).$$ Note that this two definitions never contradict each other. For all triples $a,b,c\in M_{i+1}$ for which we haven't yet defined the validity of $S(a,b,c)$ in $M_{i+1}$ by the previous two clauses we put $M_{i+1}\models S(a,b,c)$.

Let $k_0<\ldots<k_{l_i-1}$ be all the elements of $M_i$ that are not in the range of $g_i$. We define $(S,\in)$-structures on Henkin constants $E_{i,0}\subseteq \ldots\subseteq E_{i,l_i}$ together with $S$-embeddings $h_{i,j}\colon E_{i,j}\to M_{i+1}$. We put $E_{i,0}=\dot D_i$ and we define $E_{i,j+1}$ from $E_{i,j}$ as follows. We consider a fresh element $v$ and define the unique $(S,\in)$-structure $\dot E_{i,j}\gtrdot_vE_{i,j}$ by requiring the following map $\dot h_{i,j}$ to be an $S$-embedding:
$$\dot h_{i,j}\colon \dot E_{i,j}\to M_{i+1},\;\;\;\; \dot h_{i,j}(a)=h_{i,j}(a)\text{, for }a\in E_{i,j},\;\;\;\;\text{and}\;\;\;\;\dot h_{i,j}(v)=k_{j}.$$ The structure $E_{i,j+1}$ is the isomorphic copy of $\dot E_{i,j}$, where we replace $v$ with $e(E_{i,j},\dot E_{i,j})$. Note that $e(E_{i,j},\dot E_{i,j})=e(E_{i,j},E_{i,j+1})$. The function $h_{i,j+1}$ extends $h_{i,j}$ by putting $h_{i,j+1}(e(E_{i,j},\dot E_{i,j}))=k_j$. It is easy to prove by induction on $j$ using the property of $U_{i+1}$ that $\mathsf{Diag}(E_{i,j})\subseteq U_{i+1}$ and that $h_{i,j}$ are $S$-embeddings. We put $D_{i+1}=E_{i,l_i}$ and we put $g_{i+1}=h_{i,l_i}$. The order $\prec_{i+1}$ is an end-extension of $\prec_{i}$, where $$D_i\prec_{i+1}c\prec_{i+1}e(E_{i,0},E_{i,1})\prec_{i+1} e(E_{i,1},E_{i,2})\prec_{i+1}\ldots\prec_{i+1}e(E_{i,l_i-1},E_{i,l_i}).$$ We put $n_{i+1}=n_i+1$ and we put $p_{i+1,n_i}=c$. 

This finishes the description of the sequences $(M_i)_{i<\omega},(D_i)_{i<\omega},(g_i)_{i<\omega}$. Clearly, the definition gives us computable sequences. By a straightforward induction on $i$ we verify that this sequences satisfy all the local properties \ref{l1}--\ref{l6}. For the rest of the proof we verify the global properties \ref{g1},\ref{g2}, and \ref{g3}. 

We will prove by induction on $m$ that all the sequence $(p_{i,m})_{i<\omega}$ eventually stabilize to some value $p_{\infty,m}\ne \infty$, i.e. for all large enough $i$, the value $p_{i,m}$ is defined and is equal to $p_{\infty,m}$. From the construction it is clear that as soon as the  sequences $(p_{i,m'})_{i<\omega}$ stabilizes to $p_{\infty,m'}$, for all $m'<m$ the values $D_{i}{\upharpoonright \prec_i} p_{i,m}$ stabilize to some $(S,\in)$ structure $H_m$. And the restrictions $g_i{\upharpoonright \prec_i} p_{i,m}$ from the same stage stabilize to some injection $r_m$ mapping $H_m$ to an initial segment of naturals. Also it is clear that for this $H_m$ we have $U\vdash \bigwedge \mathsf{Diag}(H_m)$.

Now we proceed with the proof by induction. Assume that for all $m'<m$ the sequences  $(p_{i,m'})_{i<\omega}$ stabilize to $p_{\infty,m'}\ne \infty$ and claim that $(p_{i,m})_{i<\omega}$ stabilizes to some $p_{\infty,m}\ne \infty$. We claim that $p_{\infty,m}$ is the $<$-least Henkin constant $c$ that is not $U$-provably equal to any element of $H_m$. Let $\dot H_m$ be the unique $(S,\in)$-structure whose domain expands the domain of $H_m$ by $c$ such that $U\vdash \bigwedge \mathsf{Diag}(\dot H_m)$; the fact that $\dot H_m$ exists follows from Lemma \ref{only_good}. Using the fact that $U$ is the limit of $(U_i)_{i<\omega}$, we find a stage $i_0$ so large that for any $i\ge i_0$:
\begin{enumerate}
\item $p_{i,m'}=p_{\infty,m}$;
\item $\mathsf{Diag}(\dot H_m)\subseteq U_i$;
\item $\mathsf{Diag}(H')\not\subseteq U_i$, for any $H'\ne \dot H_m$ with the same domain as $\dot H_m$;
\item $\mathsf{Diag}(H')\not\subseteq U_i$, for any $(S,\in)$-structure $H'$ whose domain extends $H_m$ by some constant $c'<c$ such that $c'\not\in H_m$.
\end{enumerate}
Observe that for any $i\ge i_0$, if $p_{i,m}=c$, then for all $i'>i$ we also have $p_{i',m}=c$. Hence if $p_{i_0,m}=c$, then we are done with the verification of the induction step. So further we assume that $p_{i_0,m}\ne c$. If $n_{i_0}>m$, then it is trivial to see that the transition from $i_0$ to $i_0+1$ would follow the case \ref{step_o1} and $n_{i_0+1}=m$. Hence either $n_{i_0}=m$ or $n_{i_0+1}=m$. Since $i_0$ was just large enough number, we could assume that $n_{i_0}=m$ by, if necessary, replacing $i_0$ with $i_0+1$. Now we see that at the transition from $i_0$ to $i_0+1$ we  follow the case \ref{step_o2} and assign $n_{i_0+1}=m+1$, $p_{i_0+1,m}=c$. This finishes the inductive proof.

The property \ref{g1} follows from the facts that the sizes of $H_m$ monotonically increase with $m$ and the values $g_i{\upharpoonright \prec_i} p_{i,m}$ stabilize to injections $r_m$ mapping $H_m$ to initial segments of naturals. The property \ref{g2} follows from the facts that $r_m$ are $S$-embeddings of $H_m$ into $M$ and that $U\vdash \bigwedge \mathsf{Diag}(H_m)$. The property \ref{g3} follows from the fact that $p_{\infty,m}$ is always the $<$-least Henkin constant such that $U$ doesn't prove it to be equal to any $c$ from $H_m$.\end{proof}

\section{Conclusions}
Note that if a theory $T$ is a definitional extension of a theory $U$, then any model of $U$ could be uniquely expanded to a model of $T$ on the same domain. Thus for definitionally equivalent theories $T,U$ for which we fixed their joint definitional extension $V$, we have natural one-to-one correspondence between models of $T$ and models of $U$, where to transform a $T$-model $M$ to a $U$-model $M''$, we first expand $M$ to a model $M'$ of $V$ on the same domain and then put $M''$ to be the result of restriction of the signature of $M'$ to the signature of $U$.

In view of definitional equivalence of $\mathsf{PA}$ and $\mathsf{ZF}_{\mathsf{fin}}^+$ we get the following.
\begin{corollary}\label{computable_model_of_PA} There is a theory definitionally equivalent to $\mathsf{PA}$ that has a computable model such that the corresponding $\mathsf{PA}$-model is non-standard.\end{corollary}
\begin{proof}
We consider the theory $T$ that is the joint definitional extension of $T_1+\mathsf{TC}$ and $\mathsf{PA}$ obtained by adding to $T_1+\mathsf{TC}$ the same definitions of arithmetical signature that we mentioned in Section \ref{preliminaries}. Let $T_0^\star$ be the fragment of $T$ in the language with just $S$. Clearly, $T_0^\star$ is definitionally equivalent to $\mathsf{PA}$.

Let $U$ be the extension of $T$ by some false $\Sigma_1$-sentence that is not disprovable in $\mathsf{PA}$. Finally let $V$ be the fragment of $U$ in the language with just the ternary predicate $S$. Clearly, $V$ is a consistent c.e. extension of $T_0^\star\supseteq T_0$ and hence has a computable model by Theorem \ref{main_theorem}. The corresponding $\mathsf{PA}$-model is clearly non-standard.
\end{proof}

By expanding $T_0$ by the translation of the axiom of infinity to the language of $T_0$ and applying Theorem \ref{main_theorem} we get the following.
\begin{corollary}\label{computable_model_of_ZF}
There is a theory definitionally equivalent to $\mathsf{ZF}$ that has a computable model.
\end{corollary}

Note that the computable non-standard models could not be models of true arithmetic. 
\begin{theorem}\label{true_arithmetic} There are no theories definitionally equivalent to true arithmetic that have computable models corresponding to non-standard models of arithmetic.
\end{theorem}
\begin{proof}
Suppose for a contradiction that there is a theory $T$ that has a joint definitional extension $U$ with true arithmetic and a computable $T$-model $M$ such that the corresponding arithmetical model $M'$ is non-standard. 

In $U$ the formulas $x=0$, $x+1=y$, and $x\in_{\mathsf{ack}}y$ are $U$-provably equivalent to some formulas  $\mathsf{zero}(x)$, $\mathsf{suc}(x,y)$, and $x\in_{\mathsf{ack}}'y$ of the language of $T$. We transform this formula to prenex normal form, let $n$ be the maximal number of quantifier alternations in the produced prenex normal forms. Consider the formulas $\mathsf{bit}_i(x)$:
$$\exists y_0,\ldots,y_i\big(\mathsf{zero}(y_0)\land \bigwedge\limits_{j<i} \mathsf{suc}(y_{j},y_{j+1})\land y_i\in_{\mathsf{ack}}'x)\big).$$ Since $M$ is computable, the set $\{(i,a)\in \mathbb{N}\times M\mid M\models \mathsf{bit}'_i(a)\}$ is $\Sigma_{n+2}$.

We fix an arithmetical formula $\varphi(x)$ that defines a universal $\Pi_{n+2}$ set of naturals in the standard model. Clearly,$$\mathbb{N}\models \forall x\exists y\forall z<x(\varphi(z)\mathrel{\leftrightarrow}z\in_{\mathsf{ack}}y).$$ We fix some $a\in M$ that, as an element of $M'$, is a non-standard number. We find $b\in M$ such that $$M'\models \forall z<a (\varphi(z)\mathrel{\leftrightarrow}z\in_{\mathsf{ack}}b).$$ Now obviously, $$\mathbb{N}\models \varphi(i)\iff M\models \mathsf{bit}_i'(b).$$
But this gives us a $\Sigma_{n+2}$ definition for a universal $\Pi_{n+2}$ set, contradiction.
\end{proof}

However, we conjecture that the following question will have a positive answer for all $n$.
\begin{question}
Are there theories definitionally equivalent to $\mathsf{PA}+\textsf{ all true $\Pi_n$-sentences}$ that have computable non-standard models?
\end{question}

\begin{question}
Is  there a c.e. theory $T$ such that for any definitionally equivalent theory $T'$ there are no computable models?
\end{question}

\bibliographystyle{alpha}
\bibliography{bibliography}
\end{document}